\documentclass[12pt]{article}
\usepackage{fullpage,amsfonts,amssymb,amsmath,amsthm}
\newtheorem{theorem}{Theorem}[section]
\newtheorem{hypothesis}{Hypothesis}[section]
\newtheorem{corollary}{Corollary}[section]

\title{Distribution law for twin primes amongst naturals}
\date{}
\author{Boris B. Benyaminov\\
\textit{\footnotesize 6 Moa St, Belmont, North Shore},\\
\textit{\footnotesize Auckland, Postal Code: 0622, New Zealand}\\
\textit{\footnotesize boris.b.b@hotmail.com}}

\begin{document}
\maketitle
\begin{abstract}
A hypothesis is put forward regarding the function $\pi_2(x)$ which describes the distribution of twin primes in the set of natural numbers. The function $\pi_2(x)$ is tested by evaluation and an empirical $\pi_2^{\ast}(x)$ is arrived at, which is shown to be highly accurate. Several other questions are also addressed.
\end{abstract}
\section{Introduction}
In 1923, Hardy and Littlewood proposed a hypothesis for the distribution of twin primes on the interval $(1,x)$ \cite{HL}:
\begin{equation}
\pi_2(x) \sim 2\prod\limits_{p=3}^{\infty}\left(1-\frac{1}{(p-1)^2}\right)\frac{x}{\ln x}.
\label{eqn1}
\end{equation}
Later on, the following expression was put forward \cite{Luv}:
\begin{equation}
\pi_2(x) \sim 2\prod\limits_{p=3}^{\infty}\left(1-\frac{1}{(p-1)^2}\right)\frac{x}{\left(\ln x\right)^2}\prod\limits_{p=3}^{\infty}\frac{p-1}{p-2}.
\label{eqn2}
\end{equation}
The asymptotic representations (\ref{eqn1}) and (\ref{eqn2}) of the very important function $\pi_2(x)$ are too complicated to be used in practice. In this article I propose a new law for the distribution of twin primes among the naturals in the form of a much simpler $\pi_2(x)$, based on composition with the function $\pi(x)$.
\section{Results}
Recall that $\pi(x)$ $\left[\pi_2(x)\right]$ is the number of primes (twin primes) not larger than $x$. The following hypothesis is proposed for the distribution of twin primes in the set of all naturals:
\begin{hypothesis}
Twin primes are distributed among prime numbers in the same way that primes are distributed among naturals. In other words,
\begin{equation} 
\pi_2(x) = \pi\left(\pi(x)\right).
\label{eqn3}
\end{equation}
\end{hypothesis}
Table \ref{tbl1} gives some values of the functions $\pi_2(x)$ and $\pi\left(\pi(x)\right)$ for $x\leq10^6$. The values of $\pi_2(x)$ were computed according to Lehmer's tables \cite{Lehmer}. From the results in Table \ref{tbl1} it is safe to say that the ratio of $\pi_2(x)$ to $\pi\left(\pi(x)\right)$ is either exactly one or differs from unity by some negligibly small amount.\\   
\begin{table}[h]
\begin{center}
\begin{tabular}{|c|c|c|c|c|}
 \hline
  $x$ & $\pi(x)$ & $\pi_2(x)$ & $\pi\left(\pi(x)\right)$ & $\frac{\pi_2(x)}{\pi\left(\pi(x)\right)}$\\
  \hline\hline
  $25$  & $9$  & $4$   & $4$ & $1$   \\
  $50$  & $15$  & $6$   & $6$ & $1$   \\
  $75$  & $21$  & $8$   & $8$ & $1$   \\
  $125$  & $30$  & $10$   & $10$ & $1$   \\
  $150$  & $35$  & $12$   & $11$ & $1.091$   \\
  $200$  & $46$  & $15$   & $14$ & $1.071$   \\
  $300$  & $62$  & $19$   & $18$ & $1.056$   \\
  $400$  & $78$  & $21$   & $21$ & $1$   \\
  $500$  & $95$  & $24$   & $24$ & $1$   \\
  $700$  & $125$  & $30$   & $30$ & $1$   \\
  $900$  & $154$  & $35$   & $36$ & $0.972$   \\
  $1350$  & $217$  & $46$   & $47$ & $0.979$   \\
  $1500$  & $239$  & $49$   & $52$ & $0.942$   \\
  $2000$  & $303$  & $60$   & $62$ & $0.968$   \\
  $3000$  & $430$  & $81$   & $82$ & $0.988$   \\
  $4000$  & $550$  & $102$   & $101$ & $1.010$   \\
  $5000$  & $669$  & $123$   & $121$ & $1.016$   \\
  $10,000$  & $1,226$  & $201$   & $201$ & $1$   \\
  $15,000$  & $1,754$  & $268$   & $273$ & $0.982$   \\
  $20,000$  & $2,262$  & $338$   & $335$ & $1.009$   \\
  $25,000$  & $2,762$  & $403$   & $402$ & $1.002$   \\
  $30,000$  & $3,245$  & $462$   & $457$ & $1.011$   \\
  $40,000$  & $4,203$  & $585$   & $575$ & $1.017$   \\
  $50,000$  & $5,133$  & $697$   & $685$ & $1.018$   \\
  $100,000$  & $9,592$  & $1,224$   & $1,184$ & $1.034$   \\
  $200,000$  & $17,984$  & $2,159$   & $2,062$ & $1.047$   \\
  $500,000$  & $41,538$  & $4,343$   & $4,343$ & $1.035$   \\
  $1,000,000$  & $78,498$  & $7,902$   & $7,902$ & $1.033$   \\
 \hline
\end{tabular}
\caption{Testing hypothesis (\ref{eqn3}) for a few selected values of $x\leq10^6$.}
\label{tbl1}
\end{center}
\end{table} 
\\
Next, we will require the upper and lower bounds for $\pi(x)$ given in \cite{Trost}:
\begin{equation} 
\frac{2x}{3\ln x} < \pi(x) < \frac{8x}{5\ln x}.
\label{eqn4}
\end{equation}
\begin{theorem}
\label{thm1}
For all $x\geq5$ for which (\ref{eqn4}) holds, we have
\begin{equation} 
A < \pi_2(x) = \pi\left(\pi(x)\right) < B,
\label{eqn5}
\end{equation}
where
\begin{eqnarray}
A &=& \frac{4x}{9\ln x\left[\ln x-\ln(\ln x)-\ln 1.5\right]}\nonumber\ ,\\
B &=& \frac{64x}{25\ln x\left[\ln x-\ln(\ln x)+\ln 1.6\right]}\nonumber\ .
\end{eqnarray}
\end{theorem}
\begin{proof}
The function $f(x)=x/\ln x$ is monotonically increasing for $x\geq3$. From (\ref{eqn4}) we have
\begin{equation} 
\frac{2\pi(x)}{3\ln\left[\pi(x)\right]} < \pi_2(x) = \pi\left(\pi(x)\right) < \frac{8\pi(x)}{5\ln\left[\pi(x)\right]}.
\label{eqn6}
\end{equation}
If we now consider the expression $\frac{\pi(x)}{\ln\left[\pi(x)\right]}$ as a function of $\pi(x)$, we can see it is also monotonically increasing. In our case, $\pi(x)\geq3$ as $x\geq5$ (by theorem requirements). Taking this into consideration and using the right-hand side of inequality (\ref{eqn4}) 
\begin{equation}
\pi_2(x) = \pi\left(\pi(x)\right) < \frac{8\pi(x)}{5\ln\left[\pi(x)\right]} < \frac{8\frac{8x}{5\ln x}}{5\ln\left(\frac{8x}{5\ln x}\right)} =  \frac{64x}{25\ln x\left[\ln x -\ln(\ln x)+\ln1.6\right]} = B.\nonumber
\end{equation}
Similarly, by using the left-hand side of inequality (\ref{eqn4}) we obtain the lower bound:
\begin{equation}
\pi_2(x) = \pi\left(\pi(x)\right) > \frac{2\pi(x)}{3\ln\left[\pi(x)\right]} > \frac{2\frac{2x}{3\ln x}}{3\ln\left(\frac{2x}{3\ln x}\right)} =  \frac{4x}{9\ln x\left[\ln x -\ln(\ln x)-\ln1.5\right]} = A.\nonumber
\end{equation}
Thus, from (\ref{eqn6}) we have
\begin{equation}
A < \frac{2\pi(x)}{3\ln\left[\pi(x)\right]} < \pi_2(x) = \pi\left(\pi(x)\right) < \frac{8\pi(x)}{5\ln\left[\pi(x)\right]} < B,\nonumber
\end{equation}
precisely inequality (\ref{eqn5}), as required.
\end{proof}
In Table \ref{tbl2} we check inequality (\ref{eqn5}) for several values of $x$.\\
\begin{table}[h]
\begin{center}
\begin{tabular}{|c|c|c|c|}
 \hline
  $x$ & $A$ & $\pi_2(x)$ & $B$ \\
  \hline\hline
  $50$  & $3$  & $6$  & $11$    \\
  $125$  & $4$  & $10$  & $18$    \\
  $200$  & $5$  & $15$  & $23$    \\
  $300$  & $7$  & $19$  & $31$    \\
  $400$  & $8$  & $21$  & $36$    \\
  $500$  & $9$  & $24$  & $42$    \\
  $700$  & $11$  & $30$  & $53$    \\
  $1,000$  & $14$  & $35$  & $67$    \\
  $5,000$  & $44$  & $123$  & $219$    \\
  $10,000$  & $73$  & $201$  & $372$    \\
  $25,000$  & $148$  & $403$  & $762$    \\
  $50,000$  & $256$  & $697$  & $1,328$    \\
  $100,000$  & $445$  & $1,224$  & $2,331$    \\
  $500,000$  & $1,700$  & $4,494$  & $8,853$    \\
  $1,000,000$  & $2,983$  & $8,164$  & $15.887$    \\
 \hline
\end{tabular}
\caption{Testing inequality (\ref{eqn5}) for a few selected values of $x\leq10^6$.}
\label{tbl2}
\end{center}
\end{table}
\\
Next let us look at the \textit{density} of twin primes among the primes.
\begin{theorem} 
\label{thm2}
Almost all primes are not twins, so
\begin{equation}
\pi_2(x) = o\left(\pi(x)\right).
\label{eqn7}
\end{equation}
\end{theorem}
\begin{proof}
Assume that hypothesis (\ref{eqn3}) is true. Then, denoting $y=\pi(x)$, we have
\begin{equation}
0 \leq \frac{\pi_2(x)}{\pi(x)} = \frac{\pi\left(\pi(x)\right)}{\pi(x)} = \pi(y)/y.\nonumber
\end{equation}
We can find an upper bound for $\pi(y)/y$ by the sieve method, taking the set $\left\{y\right\}$ to contain no repeated values.\\
\\
Let $\varphi(y,r)$ be the number of naturals no larger than $y$ and not divisible by any of the first $r$ primes $P_1,P_2,\ldots,P_r$. Then
\begin{equation}
\varphi(y,r) = \sum\limits_{d|P_1P_2\ldots P_{\pi(\sqrt{y})}} \mu(d) \left\lfloor y/d\right\rfloor,
\label{eqn8}
\end{equation}
where $\mu(d)$ is the Mobius function and $d|P_1\ldots P_r$ means all $d$ not divisible by $P_1$ to $P_r$. It is clear that
\begin{equation} 
\pi(y) \leq \varphi(y,r)+r.
\label{eqn9}
\end{equation}
We next drop the floor operator in (\ref{eqn8}), and note that there are $2^r$ terms being summed. This means that the resulting expression has an error no larger than $2^r$, and by (\ref{eqn9}) we subsequently get
\begin{eqnarray}
\pi(y) &\leq& \sum\limits_{d|P_1P_2\ldots P_{\pi(\sqrt{y})}} \mu(d) \left\lfloor y/d\right\rfloor + r \leq y\times \sum\limits_{d|P_1P_2\ldots P_{\pi(\sqrt{y})}} \frac{\mu(d)}{d} + r + 2^r \nonumber\\
&=& y\prod\limits_{P\leq P_r} \left(1-P^{-1}\right) + r + 2^r < y\prod\limits_{P\leq P_r} \left(1-P^{-1}\right) + 2^{r+1},\nonumber
\end{eqnarray}
because $r<P_r<2^r$. Furthermore, from the inequality 
\begin{equation}
\prod\limits_{P\leq x}\left(1-P^{-1}\right)^{-1}>\ln x \nonumber,
\end{equation}
we find
\begin{equation}
\pi(y) < \frac{y}{\ln P_r} + 2^{r+1} < \frac{y}{\ln r} + 2^{r+1}.\nonumber
\end{equation}
Choose $r=c\ln y$, $c < 1/\ln2$. Then $2^r<y$ and
\begin{equation}
\pi(y) < \frac{y}{\ln\left[c\ln y\right]} + 2y^{c\ln2} = \frac{y}{\ln c + \ln\left[\ln y\right]} + 2y^{c\ln2},\nonumber
\end{equation}
where $c\ln2<1$. Dividing through by $y$, we get
\begin{equation}
0 \leq \frac{\pi(y)}{y} < \frac{1}{\ln c + \ln(\ln y)} + \frac{2}{y^{1-c\ln2}}.\nonumber
\end{equation}
As $y\rightarrow\infty$, the right-hand side of the above goes to zero, which implies the validity of equation (\ref{eqn7}).
\end{proof}
\begin{corollary}
Since we know that $\pi(x)=o(x)$, from Theorem \ref{thm2} it follows that $\pi_2(x)=o(x)$. In fact,
\begin{equation}
\lim\limits_{x\rightarrow\infty} \frac{\pi_2(x)}{x} = \lim\limits_{x\rightarrow\infty} \frac{\pi_2(x)}{\pi(x)}\cdot \lim\limits_{x\rightarrow\infty} \frac{\pi(x)}{x} = 0.\nonumber
\end{equation}
\end{corollary}
Next, we move on to construct an empirical function for the law of distribution of twin primes.\\
\\
Denote by $\eta_P$ the density of primes in the reals, and by $\eta_{PP}$ the density of twin primes in the primes, i.e. $\eta_P=\pi(x)/x$ and $\eta_{PP}=\pi_2(x)/\pi(x)$. Based on $\pi(x)=o(x)$ and (\ref{eqn7}), the densities $\eta_P$ and $\eta_{PP}$ go to zero as $x\rightarrow\infty$, but the ratio
\begin{equation}
h = \frac{\eta_{PP}}{\eta_P}
\end{equation}
remains bounded in a well-defined, constant interval (see Table \ref{tbl3}). We can obtain a rough estimate of an upper bound for $h>0$; for this we need the inequality $\pi(x)>x/\ln x$ and the right-hand side of (\ref{eqn5}). We get
\begin{eqnarray}
h &=& \frac{x\pi_2(x)}{\left[\pi(x)\right]^2} < \frac{64x^2}{\left(\frac{x}{\ln x}\right)^2 25\ln x \left(\ln x -\ln(\ln x) + \ln1.6\right)}\nonumber\\
&=& \frac{2.56\ln x}{\ln x -\ln(\ln x) + \ln1.6} < \frac{2.56\ln x}{\ln x-\ln(\ln x)} < 5.12.\nonumber
\end{eqnarray}
Thus, $0<h<5.12$. This fact allows one to construct an empirical function $\pi_2^{\ast}(x)$ for the number of twin primes on $(2,x)$. As is evident from Table \ref{tbl3}, $\pi_2^{\ast}(x)$ defined below is rather accurate.\\
\\
We obtain $\pi_2^{\ast}(x)/\pi(x) = h_c\cdot \pi(x)/x$, leading to 
\begin{equation}
\pi_2^{\ast}(x) = \left[\frac{h_c\pi^2(x)}{x}\right],
\label{eqn11}
\end{equation}
where $h_c=1.325067\ldots$ -- the mean value of $h$ for $x\leq10^6$ and we round the right-hand side of (\ref{eqn11}) to the nearest integer.\\
\\
In Table \ref{tbl3} I test the accuracy of $\pi_2^{\ast}(x)$ for $50\leq x\leq10^6$. Nevertheless, (\ref{eqn11}) is applicable for $x\geq10^6$, too. For example, there are $183,728$ twin primes less than or equal to $x=37\cdot10^6$, while $\pi_2^{\ast}(x)=183,463$ which gives a relative error of $\delta=0.0014$ (see Table \ref{tbl3}). 

\begin{table}[h]
\begin{center}
\begin{tabular}{|c|c|c|c|c|c|}
 \hline
  $x$ & $h$ & $\pi_2(x)$ & $\pi_2^{\ast}(x)$ & $\left|\Delta\right|$ & $\frac{\left|\Delta\right|}{\pi_2(x)}$ \\
  \hline\hline
  $50$ & $1.333336$ & $6$ & $6$ & $0$ & $0$ \\
  $150$ & $1.346938$ & $11$ & $11$ & $0$ & $0$ \\
  $500$ & $1.329639$ & $24$ & $24$ & $0$ & $0$ \\
  $1,500$ & $1.286742$ & $49$ & $50$ & $1$ & $0.0204$ \\
  $2,000$ & $1.307061$ & $60$ & $61$ & $1$ & $0.0167$ \\
  $3,000$ & $1.314223$ & $81$ & $82$ & $1$ & $0.0123$ \\
  $4,000$ & $1.348760$ & $102$ & $100$ & $2$ & $0.0196$ \\
  $5,000$ & $1.374114$ & $123$ & $119$ & $4$ & $0.0325$ \\
  $10,000$ & $1.330737$ & $201$ & $200$ & $1$ & $0.0050$ \\
  $15,000$ & $1.306672$ & $268$ & $274$ & $6$ & $0.0224$ \\
  $20,000$ & $1.321178$ & $338$ & $339$ & $1$ & $0.0030$ \\
  $25,000$ & $1.320680$ & $403$ & $404$ & $1$ & $0.0025$ \\
  $30,000$ & $1.316236$ & $462$ & $465$ & $3$ & $0.0065$ \\
  $40,000$ & $1.324637$ & $585$ & $585$ & $0$ & $0$ \\
  $50,000$ & $1.322696$ & $697$ & $698$ & $1$ & $0.0014$ \\
  $100,000$ & $1.330341$ & $1,224$ & $1,219$ & $5$ & $0.0041$ \\
  $200,000$ & $1.335088$ & $2,159$ & $2,143$ & $16$ & $0.0074$ \\
  $500,000$ & $1.302302$ & $4,494$ & $4,573$ & $79$ & $0.0176$ \\
  $1,000,000$ & $1.342908$ & $8,164$ & $8,165$ & $1$ & $0.0001$ \\
 \hline
\end{tabular}
\caption{Testing expression (\ref{eqn11}) for a few selected values of $x\leq10^6$.}
\label{tbl3}
\end{center}
\end{table}

\subsection*{Acknowledgements}
The author would like to thank Sophie S. Shamailov for translating and typing the article.

\end{document}